\documentclass[12pt]{amsart}

\usepackage{amssymb}

\newtheorem{Thm}{Theorem}[section]

\newtheorem{problem}[Thm]{Problem}

\newtheorem{lemma}[Thm]{Lemma}
\newtheorem{proposition}[Thm]{Proposition}
\newtheorem{definition}[Thm]{Definition}
\newtheorem{remark}[Thm]{Remark}

\newtheorem{theorem}[Thm]{Theorem}

\newtheorem*{kadison-singer}{Kadison-Singer Problem}
\newtheorem*{paving conjecture}{Paving Conjecture}
\newtheorem*{bourgain-tzafriri}{Bourgain-Tzafriri Conjecture}
\newtheorem*{feichtinger conjecture}{Feichtinger Conjecture}

\def\ldots{\mathinner{\ldotp\ldotp\ldotp}}
\def\ldots{\mathinner{\cdotp\cdotp\cdotp}}

\def \H{{\mathbb H}}
\def \Z{{\mathbb Z}}
\def \N{{\mathbb N}}
\def \DD{{\mathbb D}}
\def \CC{{\mathbb C}}
\def \B(l_2){B(\ell_2)}

\def \K{{\mathbb K}}

\def \cal{\mathcal}

\def \beq{\begin{eqnarray*}}
\def \eeq{\end{eqnarray*}}

\def \< {\langle}
\def \> {\rangle}

\def \R {\mathbb{R}}

\def \<{\langle}
\def \>{\rangle}

\begin{document}

\title[The Solution to the Kadison-Singer Problem:  Consequences]{Consequences
of the Marcus/Spielman/Srivastava Solution of the Kadison-Singer Problem}
\author[P.G. Casazza and J.C. Tremain
 ]{Peter G. Casazza and Janet C. Tremain}

\address{Department of Mathematics \\
University of Missouri-Columbia \\
Columbia, MO 65211}
\email{casazzap@missouri.edu; tremainjc@missouri.edu}

\thanks{The authors were supported by NSF DMS 1307685; NSF ATD 1042701 and 1321779; AFOSR  DGE51:  FA9550-11-1-0245.}


\begin{abstract}
It is known that the famous, intractible 1959 
Kadison-Singer problem in $C^{*}$-algebras
is equivalent to fundamental unsolved problems in a dozen areas
of research in pure mathematics, applied mathematics and Engineering.  The recent surprising solution to this
problem by Marcus, Spielman and Srivastava was a significant achievement and a significant advance for all these areas of research.
  We will look at many of the known
equivalent forms of the Kadison-Singer Problem and
see what are the best new theorems available in each
area of research as a consequence of the work of
Marcus, Spielman and Srivastave.  In the cases 
where {\it constants}
are important for the theorem, we will give the best
constants available in terms of a {\it generic constant}
taken from \cite{MSS}.  Thus, if better constants eventually become available, it will be simple to
adapt these new constants to the theorems.  
\end{abstract}

\maketitle

\section{Introduction}\label{Intro}
\setcounter{equation}{0}

The famous 1959 Kadison-Singer Problem \cite{KS} has defied the best efforts
of some of the most talented mathematicians of 
the last 50 years.  The recent solution to this
problem by Marcus, Spielman and Srivastave \cite{MSS} is 
not only a significant mathematical achievement by
three very talented mathematicians, but it is also a
major advance for a dozen different areas of
research in pure mathematics, applied mathematics
and engineering.

\begin{kadison-singer}[KS]
Does every pure state on the (abelian) von Neumann algebra $\DD$ of
bounded diagonal operators on ${\ell}_2$ have a unique extension to
a (pure) state on $B({\ell}_2)$, the von Neumann algebra of all
bounded linear operators on the Hilbert space ${\ell}_2$?
\end{kadison-singer}

A {\bf state} of a von Neumann algebra ${\cal R}$ is a linear
functional $f$ on ${\cal R}$ for which $f(I) = 1$ and $f(T)\ge 0$
whenever $T\ge 0$ (whenever $T$ is a positive operator).
The set of states of ${\cal R}$ is a convex subset of the dual space
of ${\cal R}$ which is compact in the ${\omega}^{*}$-topology.  By the
Krein-Milman theorem, this convex set is the closed convex hull of its
extreme points.  The extremal elements in the space of states are
called the {\bf pure states} (of ${\cal R}$).

This problem arose from the very productive
collaboration of Kadison and Singer in
the 1950's which culminated in their seminal work on triangular
operator algebras.  During this collaboration, they often discussed
the fundamental work of Dirac \cite{Di} on Quantum Mechanics.  
In particular, they kept returning to one part of Dirac's work because it
seemed to be problematic.   Dirac wanted to find a
``representation'' (an orthonormal basis) for a compatible
family of observables (a commutative family of self-adjoint
operators).  On pages 74--75 of \cite{Di} Dirac states:

\vskip8pt
\hskip.5truein\vbox{\hsize4truein
``To introduce a representation in practice
\vskip8pt
(i)  \hskip4pt We look for observables which we would like to have diagonal
either because we are interested in their probabilities or for
reasons of mathematical simplicity;

(ii)  \hskip4pt We must see that they all commute --- a necessary condition
since diagonal matrices always commute;

(iii)  \hskip4pt We then see that they form a complete commuting set, and if
not we add some more commuting observables to make them into a complete
commuting set;

(iv)  \hskip4pt We set up an orthogonal representation with this commuting set
diagonal.
\vskip8pt

{\bf The representation is then completely determined} ...
 {\bf by the observables
that are diagonal} ...''}
\vskip6pt
The emphasis above was added.  Dirac then talks about finding a basis that
diagonalizes a self-adjoint operator, which is troublesome since there
are perfectly respectable self-adjoint operators which do not have
a single eigenvector.  Still, there is a {\it spectral resolution}
of such operators.  Dirac addresses this problem on pages 57-58
of \cite{Di}:
\vskip8pt
\hskip.5truein\vbox{\hsize4truein
``We have not yet considered the lengths of the basic vectors.  With an
orthonormal representation, the natural thing to do is to normalize
the basic vectors, rather than leave their lengths arbitrary, and
so introduce a further stage of simplification into the representation.
However, it is possible to normalize them only if the parameters are
continuous variables that can take on all values in a range, the basic
vectors are eigenvectors of some observable belonging to eigenvalues
in a range and are of infinite length...''}
\vskip8pt

In the case of $\DD$, the representation is
$\{e_i\}_{i\in I}$, the orthonormal basis of $l_2$.  But what happens if our
observables have ``ranges" (intervals) in their spectra?  This led Dirac to
introduce his famous $\delta$-function --- vectors of ``infinite length."
From a mathematical point of view, this is problematic.  What we need is to
replace the vectors $e_i$ by some mathematical object that is essentially the
same as the vector, when there is one, but gives us something precise and
usable when there is only a $\delta$-function.  This leads to the ``pure
states'' of $\B(l_2)$ and, in particular, the (vector) pure states $\omega_x$,
given by $\omega_x(T)=\<Tx,x\>$, where $x$ is a unit vector in $\H$.  Then,
$\omega_x(T)$ is the expectation value of $T$ in the state corresponding to
$x$.  This expectation is the average of values measured in the laboratory
for the ``observable" $T$ with the system in the state corresponding to $x$.
The pure state $\omega_{e_i}$ can be shown to be completely determined by its
values on $\DD$; that is, each $\omega_{e_i}$ has a {\it unique\/} extension
to $\B(l_2)$.  But there are many other  pure states of $\DD$.  (The family of
all pure states of $\DD$ with the $w^*$-topology is $\beta(\Z)$, the
$\beta$-compactification of the integers.)  Do these other pure states have
unique extensions?  This is the Kadison-Singer problem (KS).

By a ``complete" commuting set, Dirac means what is now called a ``maximal
abelian self-adjoint" subalgebra of $\B(l_2)$; $\DD$ is one such.  There are
others.  For example, another is generated by an observable whose``simple"
spectrum is a closed interval.  Dirac's claim, in mathematical form, is that
each pure state of a ``complete commuting set" has a unique state extension
to $\B(l_2)$.  Kadison and Singer show [37] that that is {\it not so\/} for
each complete commuting set other than $\DD$.  They also show that each pure
state of $\DD$ has a unique extension to the uniform closure of the algebra
of linear combinations of operators $T_\pi$ defined by $T_\pi e_i=e_{\pi(i)}$,
where $\pi$ is a permutation of $\Z$.

Kadison and Singer believed that KS had a negative answer.  In particular,
on page 397 of \cite{KS} they state:  ``We incline to the view
that such extension is non-unique''. 

Over the 55 year history of the Kadison-Singer Problem, a significant amount
of research was generated resulting in a number of partial results as well as a
large number of equivalent problems.  These include the {\bf Anderson Paving
Conjectures} \cite{A,A3,A2}, the {\bf Akemann-Anderson Projection Paving
Conjecture} \cite{AA}, the {\bf Weaver Conjectures} \cite{W},
the {\bf Casazza-Tremain Conjecture} \cite{CT}, the
{\bf Feichtinger Conjecture} \cite{CCLV}, the {\bf $R_{\epsilon}$-Conjecture}
\cite{CT}, the {\bf Bourgain-Tzafriri Conjecture} \cite{CT}, the {\bf Sundberg Problem}
 \cite{CKBook}.  Many directions for approaching this problem were proposed and
solutions were given for special cases:  All matrices with positive coefficients
are pavable \cite{HKW} as are all matrices with "small" coefficients \cite{BT1}.
Under stronger hypotheses, solutions to the problem were given by 
Berman/Halpern/Kaftal/Weiss \cite{BHKW}, Baranov and Dyakonov \cite{BD},
Paulsen \cite{LP,P1,P}, Lata \cite{L}, Lawton \cite{La}, Popa \cite{Po},
Grochenig \cite{G}, Bownik/Speegle \cite{BS}, Casazza/Christensen/Lindner/Vershynin
\cite{CCLV}, Casazza/Christensen/Kalton \cite{CCK}, Casazza/Kutyniok/Speegle \cite{CKS2},
Casazza/Edidin/Kalra/Paulsen \cite{CEKP} and much
more.

Our goal in this paper is to see how the solution of \cite{MSS} to the Kadison-Singer
Problem answers each of the above problems in a quantative way and to compute
the best available constants at this time.

The paper is organized as follows.  In Section \ref{FT} we introduce the basics
of Hilbert space frame theory which forms the foundation for producing equivalences
of the Paving Conjecture. Next, in Section \ref{MSSW} we give the basic
Marcus/Spielman/Srivastava result proving {\bf Weaver's Conjecture}.  
In Section \ref{MSSPC} we present their
proof of the {\bf Akemann-Anderson Projection Paving Conjecture} and 
the {\bf Anderson Paving Conjecture}.  In Section \ref{EPC} we prove the
{\bf Casazza/Tremain Conjecture}, the {\bf Feichtinger Conjecture}, the
{\bf $R_{\epsilon}$-Conjecture}, and the {\bf Bourgain-Tzafriri Conjecture}.
In Section \ref{HA} we prove the Feichtinger Conjecture in {\bf Harmonic
Analysis} (the stronger form involving {\bf syndetic sets}),
and solve the {\bf Sundberg Problem}.  Section \ref{LDS} contains equivalents of the Paving Conjecture
for {\bf Large and Decomposable subspaces} of a Hilbert space.  Finally,
in Section \ref{A} we will trace some of the history of the Paving Conjecture.

\section{Frame Theory}\label{FT}
\setcounter{equation}{0}

Hilbert space {\bf frame theory} is the tool which is used to connect many of
the equivalent forms of the {\bf Paving Conjecture}.  So we start with
an introduction to this area.
A family of vectors $\{f_i\}_{i\in I}$ in a Hilbert space $\H$
is a {\bf Riesz basic sequence} if there are constants $A,B>0$ so that
for all scalars $\{a_i\}_{i\in I}$ we have:
\begin{displaymath}
A\sum_{i\in I}|a_i |^2 \le \|\sum_{i\in I}a_i f_i \|^2 \le
B\sum_{i\in I}|a_i|^2.
\end{displaymath}
We call $A,B$ the {\bf lower and upper Riesz basis
bounds} for $\{f_i\}_{i\in I}$.  If the Riesz
basic sequence $\{f_i\}_{i\in I}$ spans $\H$
we call it a {\bf Riesz basis} for $\H$.  So $\{f_i\}_{i\in I}$
is a Riesz basis for $\H$ means there is an orthonormal basis
$\{e_i\}_{i\in I}$ so that the operator $T(e_i )=f_i$ is invertible.
In particular, each Riesz basis is {\bf bounded}.  That is,
$0 < \inf_{i\in I}\|f_i\| \le \sup_{i\in I}\|f_i\| < \infty$. 

Hilbert space frames were introduced by Duffin and Schaeffer
\cite{DS} to address some very deep problems in nonharmonic
Fourier series (see \cite{Y}).
A family $\{f_i \}_{i\in I}$ of elements of a (finite
or infinite dimensional) Hilbert space
$\H$ is called a {\bf frame} for $\H$ if
there are constants $0<A\le B < \infty$ (called the
{\bf lower and upper frame bounds}, respectively)
so that for all $f\in \H$
\begin{equation}\label{E5}
A\|f\|^2 \le \sum_{i\in I}|\langle f,f_i \rangle |^2
\le
B\|f\|^2.
\end{equation}
If we only have the right hand inequality in Equation \ref{E5}
 we call $\{f_i\}_{i\in I}$
a {\bf Bessel sequence with Bessel bound B}.
  If $A=B$, we call this an
$A$-{\bf tight frame} and if $A=B=1$, it is called a
{\bf Parseval frame}.  If all the frame elements have the same
norm, this is an {\bf equal norm} frame and if the
frame elements are of unit norm, it is a {\bf unit norm
frame}.  It is immediate that $\|f_i\|^2\le B$.  If also
inf $\|f_i\|>0$, $\{f_i\}_{i\in I}$ is a {\bf bounded frame}.
The numbers $\{\langle f,f_i\rangle\}_{i\in I}$ are
 the {\bf frame coefficients} of the vector $f\in \H$.
If $\{f_i \}_{i\in I}$ is a Bessel sequence,
the {\bf synthesis operator} for $\{f_i\}_{i\in I}$
is the bounded linear operator $T:{\ell}_{2}(I) \rightarrow
\H$ given by $T(e_i )= f_i$ for all $i\in I$.
The
{\bf analysis operator} for $\{f_i\}_{i\in I}$
is $T^{*}$ and satisfies: $T^{*}(f) =
\sum_{i\in I} \langle f,f_i \rangle e_i$.  In particular,
$$
\|T^{*}f\|^2 = \sum_{i\in I}|\langle f,f_i\rangle |^2,\ \
\mbox{for all $f\in \H$},
$$
and hence the smallest Bessel bound for $\{f_i\}_{i\in I}$
equals $\|T^{*}\|^2=\|T\|^2$.
Comparing this to Equation \ref{E5} we have:

\begin{theorem}\label{FTTT}
Let $\H$ be a Hilbert space and $T:{\ell}_2(I)\rightarrow
\H$, $Te_i = f_i$ be a bounded linear operator.  The following
are equivalent:
\begin{enumerate}
\item $\{f_i\}_{i\in I}$ is a frame for $\H$.
\vspace*{.04cm}
\item The operator $T$ is bounded, linear, and onto.
\vspace*{.04cm}
\item The operator $T^{*}$ is an (possibly into) isomorphism.
\end{enumerate}
Moreover, if $\{f_i\}_{i\in I}$ is a Riesz basis then it is a frame and the
Riesz bounds equal the frame bounds.
\end{theorem}

It follows that a Bessel sequence
is a Riesz basic sequence if and only if $T^{*}$ is onto.
   The {\bf frame
operator} for the frame is the positive, self-adjoint invertible
operator $S=TT^{*}:\H \rightarrow \H$.  That is,
$$
Sf = TT^{*}f = T\left ( \sum_{i\in I}\langle f,f_i\rangle e_i
\right ) = \sum_{i\in I}\langle f,f_i\rangle Te_i =
\sum_{i\in I}\langle f,f_i\rangle f_i.
$$
In particular,
$$
\langle Sf,f\rangle = \sum_{i\in I}|\langle f,f_i \rangle|^2.
$$
It follows that $\{f_i\}_{i\in I}$ is a frame with frame
bounds $A,B$ if and only if $A \cdot I \le S \le B \cdot I$.
So $\{f_i\}_{i\in I}$ is a Parseval frame if and only if $S=I$.
{\bf Reconstruction} of vectors in $\H$ is achieved via
the formula:
\begin{eqnarray*}
f &=& SS^{-1}f = \sum_{i\in I}\langle S^{-1}f,f_i \rangle f_i \\
&=& \sum_{i\in I}\langle f,S^{-1}f_i \rangle f_i \\
&=& \sum_{i\in I}\langle f,f_i \rangle S^{-1}f_i \\
&=& \sum_{i\in I}\langle f,S^{-1/2}f_i \rangle S^{-1/2}f_i.
\end{eqnarray*}

Recall that for vectors $u,v\in \H$ the {\bf outer product} of these
vectors $uv^*$ is the rank one operator defined by:
\[ (uv^*)(x) = \langle x,v\rangle u.\]
In particular, if $\|u\|=1$ then $uu^*$ is the rank one projection
of $\H$ onto span u.  Also, $v^*u= \langle u,v\rangle$.
 The frame operator $S$ of the frame
$\{f_i\}_{i\in I}$ is representable as
\[ S = \sum_{i\in I}f_if_i^*.\]
The {\bf Gram operator} of the frame $\{f_i\}_{i\in I}$ is
$G=T^*T$ and has the matrix
\[ G = (\langle f_i,f_j\rangle)_{i,j\in I}= (f_j^*f_i)_{i,j\in I}.\]
It follows that the non-zero eigenvalues of $G$ equal the non-zero
eigenvalues of $S$ and hence $\|G\|=\|S\|$.

I alsot follows that $\{S^{-1/2}f_i\}_{i\in I}$ is a Parseval frame
{\bf isomorphic} to $\{f_i\}_{i\in I}$.  Two sequences
$\{f_i\}_{i\in I}$ and $\{g_i\}_{i\in I}$ in a
Hilbert space are {\bf isomorphic} if there is a well-defined invertible
operator $T$ between their spans with $Tf_i = g_i$
for all $i\in I$.  
We now show that there is a simple
way to tell when two frame sequences are isomorphic.

\begin{proposition}\label{FTP10}
Let $\{f_i\}_{i\in I}$, $\{g_i\}_{i\in I}$ be frames for a Hilbert
space $\H$ with analysis operators $T_1$ and $T_2$, respectively.
The following are equivalent:

(1)  The frames $\{f_i\}_{i\in I}$ and $\{g_i\}_{i\in I}$ are isomorphic.

(2)  ker $T_1$ = ker $T_2$.
\end{proposition}

{\it Proof}:
$(1)\Rightarrow (2)$:  If $Lf_i = g_i$ is an isomorphism, then
$Lf_i = LT_1 e_i = g_i = T_2 e_i$ quickly implies our statement
about kernels.

$(2)\Rightarrow (1)$:  Since $T_i|_{{(ker\ T_i)}^{\perp}}$ is an
isomorphism for $i=1,2$, if the kernels are equal, then
$$
T_2 \left ( T_1 |_{(ker\ T_2)^{\perp}}\right )^{-1} f_i = g_i
$$
is an isomorphism.
\qed

 In the finite
dimensional case, if
$\{g_j\}_{j=1}^{n}$ is an orthonormal basis of
${\ell}_2^n$ consisting of eigenvectors for $S$ with respective eigenvalues
$\{{\lambda}_j\}_{j=1}^{n}$, then for every $1\le j\le n$,
$\sum_{i\in I}|\langle f_i ,g_j \rangle|^2 = {\lambda}_j$.  In particular,
$\sum_{i\in I}\|f_i\|^2 =$ trace S ($=n$ if $\{f_i\}_{i\in I}$
is a Parseval frame).  An important result is

\begin{theorem}\label{FT20}
If $\{f_i\}_{i\in I}$ is a frame for $\H$ with frame bounds
$A,B$ and $P$ is any orthogonal projection on $\H$, then
$\{Pf_i\}_{i\in I}$ is a frame for $P\H$ with frame bounds
$A,B$.
\end{theorem}

{\it Proof}:
For any $f\in P\H$,
$$
\sum_{i\in I}|\langle f,Pf_i \rangle |^2 =
\sum_{i\in I}|\langle Pf,f_i \rangle |^2 = \sum_{i\in I}|\langle
f,f_i \rangle |^2.
$$
\qed

A fundamental result in frame theory was proved independently
by Naimark and Han/Larson \cite{C,HL}.  For completeness we
include its simple proof.

\begin{theorem}\label{T3}
A family $\{f_i\}_{i\in I}$ is a Parseval frame for a Hilbert
space $\H$ if
and only if there is a containing Hilbert space $\H \subset {\ell}_2 (I)$
with an orthonormal basis $\{e_i\}_{i\in I}$ so that
the orthogonal projection $P$ of ${\ell}_2 (I)$ onto $\H$ satisfies
$P(e_i) = f_i$ for all $i\in I$.
\end{theorem}

{\it Proof}:
The ``only if'' part is Theorem \ref{FT20}.  For the ``if'' part,
if $\{f_i\}_{i\in I}$ is a Parseval frame, then the synthesis operator
 $T:{\ell}_2(I) \rightarrow \H$ is a partial isometry.  So $T^{*}$
is an isometry and we can associate $\H$ with $T^{*}\H$.  Now, for
all $i\in I$ and all $g=T^{*}f \in T^{*}\H$ we have
$$
\langle T^{*}f,Pe_i \rangle = \langle T^{*}f,e_i \rangle
= \langle f,Te_i \rangle = \langle f,f_i \rangle = \langle T^{*}f,
T^{*}f_i \rangle.
$$
It follows that $Pe_i = T^{*}f_i$ for all $i\in I$.
\qed

  For an introduction
to frame theory we refer the reader to Han/Kornelson/Larson/Weber 
\cite{HKLW}, Christensen \cite{C}
and Casazza/Kutyniok \cite{CKBook}.

\section{Marcus/Spielman/Srivastave and Weaver's Conjecture}\label{MSSW}

In \cite{MSS} the authors do a deep analysis of what
 they call {\it mixed characteristic polynomials} to prove a famous conjecture of Weaver \cite{W} which
Weaver had earlier shown is equivalent to the {\it Paving
Conjecture} which Anderson \cite{A} had previously
shown was equivalent to the {\bf Kadison-Singer Problem}.  We will mearly state the main theorem
from \cite{MSS} here and use it to find the best
constants in the various equivalent forms of the
Kadison-Singer Problem.

\begin{theorem}[Marcul/Spielman/Srivastava]\label{MSS1}
Let r be a positive integer and let $u_1,u_2,\ldots,u_m \in \CC^d$ be vectors such that
\[ \sum_{i=1}^mu_iu_i^* = I,\]
and $\|u_i\|^2 \le \delta$ for all $i$.  Then there
is a partition $\{S_1,S_2,\ldots,S_r\}$ of $[m]$
such that
\[ \left \| \sum_{i\in S_j}u_iu_i^* \right \|
\le \left ( \frac{1}{\sqrt{r}}+\sqrt{\delta}
\right )^{2},\mbox{ for all }j=1,2,\ldots,r.\]
\end{theorem}

In \cite{W}, Weaver reformulated the Paving Conjecture into 
{\bf Discrepancy Theory} which generated a number of
new equivalences of the Paving Conjecture \cite{CT,CFTW}
and set the stage for the eventual solution
to the problem.
Setting $r=2$ and $\delta = 1/18$ \cite{MSS}, this implies the 
original {\bf Weaver Conjecture $KS_2$} \cite{W} with $\eta=18$ and
$\theta=2$.

\begin{theorem}[Marcus/Spielman/Srivastave]\label{MSS10}
There are universal constants $\eta \ge 2$ and $\theta
>0$ so that the following holds.  Let $u_1,u_2,\ldots,
u_m \in \CC^d$ satisfy $\|u_i\|\le 1$ for all $i$
and suppose
\[ \sum_{i=1}^M|\langle u,u_i\rangle|^2
= \eta,\mbox{ for every unit vector } u\in \CC^d.\]
Then there is a partition $S_1,S_2$ of $\{1,2,\ldots,m\}$ so that
\[ \sum_{i\in S_j}|\langle u,u_i\rangle|^2 \le 
\eta - \theta,\]
for every unit vector $u\in \CC^d$ and each $j\in \{
1,2\}.$

Moreover, $\eta=18$ and $\theta=2$ work.
\end{theorem}

To make Theorem \ref{MSS1}  more usable later, we will reformulate it into the
language of operator theory. Recall, for a matrix operator
\[ T = (a_{ij})_{i,j=1}^m,\]
we let
\[ \delta(T) = \min_{1\le i \le m}|a_{ii}|.\]

\begin{theorem}\label{CT}
Let $r$ be a positive integer.  Given
 an orthogonal projection $Q$ on $\ell_2^m$
with $\delta(Q) \le \delta$, there are diagonal projections $\{P_j\}_{j=1}^r$
with
\[ \sum_{j=1}^rP_j=I,\]
and
\[ \|P_jQP_j\|\le \left ( \frac{1}{\sqrt{r}}+\sqrt{\delta}\right )^2,
\mbox{ for all }j=1,2,\ldots,r.\] 
\end{theorem}

\begin{proof}
Letting $u_i=Qe_i$ for all $i=1,2,\ldots,m$, we have that $Q=(u_i^*u_j)_{i,j=1}^m$.
Choose a partition $\{S_j\}_{j=1}^r$ of $[m]$ satisfying Theorem \ref{MSS1}
and let $P_j$ be the diagonal projection
onto $\{e_i\}_{i\in S_j}$.
For any $k\in [r]$ we have:
\[ \|P_kQP_k\|= \|(u_i^*u_j)_{i,j\in S_k}\| = \|\sum_{i\in S_k}u_iu_i^*\|
\le  \left ( \frac{1}{\sqrt{r}}+\sqrt{\delta}\right )^2.\]
\end{proof}

\section{Marcus/Spielman/Srivastava and the Paving Conjectures}\label{MSSPC}

Perhaps the most significant advance on the Kadison-Singer Problem occured when
Anderson \cite{A} showed that it was equivalent to what became known as the
{\bf (Anderson) Paving Conjecture}.  The significance of this was that it removed
the Kadison-Singer Problem from the burden of being a very technical problem
in $C^*$-Algebras which had no real life outside the field, to making it a highly
visible problem in Operator Theory which generated a significant amount of
research - and eventually led to the solution to the problem.

\begin{definition}
Let $T:\ell_2^n \rightarrow \ell_2^n$ be an operator.
Given $r \in \N$ and $\epsilon >0$ we say that $T$ can
be {\bf $(r,\epsilon)$-paved} if there is a partition
$\{A_j\}_{j=1}^r$ of $[n]$ so that if $P_i$ is the
coordinate projection onto the coordinates $A_j$ so
that
\[ \|P_iTP_i\|\le \epsilon\|T\|,\mbox{ for all }i=1,2,\ldots,r.\]

Or equivalently, there are coordinate projections
$\{P_i\}_{i=1}^r$ so that
\[ \sum_{i=1}^r P_i =I\mbox{ and }\|P_iTP_i\|\le \epsilon \|T\|.\]
\end{definition}

A major advance was made on the Kadison-Singer Problem by
Anderson \cite{A}.

\begin{theorem}[Anderson Paving Conjecture]
The following are equivalent:

(1)  The Kadison-Singer Conjecture is true.

(2)  For every $0<\epsilon <1$ there is an $r= r(\epsilon)\in \N$ so that
every operator $T$ on a finite or infinite dimensional
Hilbert space is $(r,\epsilon)$-pavable.

(3)  For every $0<\epsilon <1$ there is an $r= r(\epsilon)\in \N$ so that
every self-adjoint operator $T=T^*$ on a finite or infinite dimensional
Hilbert space  is $(r,\epsilon)$-pavable.
\end{theorem}

This result become known as the {\bf Anderson Paving Conjecture} and became a major
advance for the field.  For the next 25 years a significant amount of effort was directed
at proving (or giving a counter-example to) the Paving Conjecture.  We give a brief
outline of the history of this effort in Section \ref{A}.

In 1991, Akemann and Anderson reformulated the Paving Conjecture for {\bf operators}
into a paving conjecture for {\bf projections}.  They also gave a number
of conjectures concerning paving projections and the Paving Conjecture.
This was a major advance for the area reducing the problem to a very
special class of operators.
Theorem \ref{MSS1} also implies the {\bf Akemann-Anderson Projection
Paving Conjecture} \cite{AA} which they showed implies a positive solution
to the Kadison-Singer Problem.

\begin{theorem}
Given $\epsilon >0$, choose $\delta >0$ so that
\[ \left ( \frac{1}{\sqrt{2}}+\sqrt{\delta}\right )^2 \le 1-\epsilon.\]
For any projection $Q$ on $\ell_2^m$ of rank d there is a
diagonal projection $P$ on $\ell_2^m$ so that
\[ \|PQP\|\le 1-\epsilon \mbox{ and } \|(I-P)Q(I-P)\|\le 1-\epsilon.\]
\end{theorem}

\begin{proof}
This is immediate from Theorem \ref{CT} letting $r=2$ and noting that
$P_2 = (I-P_1)$.
\end{proof}

The authors \cite{MSS} then give a quantative proof
of the original {\bf Anderson Paving Conjecture}
\cite{A}.  To do this, we will first look at an elementary way
to pass between paving for operators and paving for projections
introduced by Casazza/Edidin/Kalra/Paulsen \cite{CEKP}.
In \cite{CEKP} there is a simple method for passing
paving numbers back and forth between operators and
projections with constant diagonal $1/2$ (or $1/2^k$
if we iterate this result).  This was a serioius change in
direction for the paving conjecture for projections.  The
earlier work of Akemann/Anderson \cite{AA} and Weaver \cite{W}
emphasized paving for projections with small diagonal while the
results in \cite{CEKP} showed that it is more natural to work with
projections with large diagonal. 
The proof is a direct calculation.

\begin{theorem}[Casazza/Edidin/Kalra/Paulsen]\label{CEKP}
If $T$ is a self-adjoint operator with $\|T\|\le 1$
then
\[ A= \begin{bmatrix}
T & \sqrt{I-T^2}\\
\sqrt{I-T^2} & -T
\end{bmatrix}\]
is an idempotent.  I.e.  $A^2=I$.

Hence,
\[ P = \frac{I\pm A}{2},\]
is a projection.
\end{theorem}

It follows that the paving numbers for self-adjoint operators are at most the square of the paving numbers for projections.
Using Theorem \ref{CEKP}, in \cite{CEKP} they prove the following (See also
\cite{MSS}):

\begin{proposition}[Casazza/Edidin/Kalra/Paulsen]\label{prop2.33}
Suppose there is a function $r:\R_{+} \rightarrow \N$
so that every $2n \times 2n$ projection matrix $Q$ with diagonal entries equal to $1/2$ can be $(r(\epsilon),
\frac{1+\epsilon}{2})$-paved, for every $0< \epsilon<1$.
Then every $n\times n$ self-adjoint zero diagonal matrix
$T$ can be $(r^2(\epsilon),\epsilon)$-paved for all
$0<\epsilon<1$.
\end{proposition}

\begin{proof}
Given $Q$ as in the proposition, $Q=(u_i^*u_j)_{i,j\in [2n]}$ is the gram matrix of $2n$ vectors $u_1,u_2,\ldots,u_{2n}\in \CC^n$ with $\|u_i\|^2 = 1/2 = \delta$.  Applying Theorem  \ref{CT} we find a partition
$\{A_i\}_{i=1}^r$ of $[2n]$ so that if $P_i$ is the diagonal projection onto the coordinates of $A_i$ we
have for $k\in [r]$,
\[ \|P_kTP_k\| = \|(u_i^*u_j)_{i,j\in A_k}\|
= \|\sum_{i\in A_k}u_iu_i^*\| \le \left (
\frac{1}{\sqrt{r}}+ \frac{1}{\sqrt{2}} \right )^2
<\frac{1}{2}+\frac{3}{\sqrt{r}}<\frac{1+\epsilon}{2},\]
if $r= (\frac{6}{\epsilon})^2$.
So every $Q$ can be $(r,\frac{1+\epsilon}{2})$-paved. 
\end{proof}

It follows that every self-adjoint operator can be $(R\epsilon)$-paved for $r= (\frac{6}{\epsilon})^4$, in either
the real or complex case.

Using this and Theorem \cite{CT}, \cite{MSS} gives a quantative proof  of the {\bf Anderson
Paving Conjecture} and hence of the Kadison-Singer Problem.

\begin{theorem}\label{thm6.1}
For every $0<\epsilon<1$, every zero-diagonal real (Resp. complex)
self-adjoint matrix $T$ can be $(r,\epsilon)$-paved
with $r = (6/\epsilon)^4$ (Resp. $r=(6/\epsilon)^8$).
\end{theorem}

\noindent {\bf Important}:
For complex Hilbert spaces, given an operator
$T$, we write it as $T=A+iB$ where $A,B$ are real
operators.  Paving $A,B$ separately and intersecting the paving
sets, we have a paving of $T$ but with the paving number
squared.

\begin{remark}
In \cite{CEKP} it is shown that $1/\epsilon^2 \le r$
in Theorem \ref{thm6.1}.  We can compare this to the value
$r= (\frac{6}{\epsilon})^4$ we are getting from the theorem.
\end{remark}

\section{Equivalents of the Paving Conjecture}\label{EPC}
 
 Casazza/Tremain reformulated the Paving Conjecture into a number
of conjectures related to problems in engineering.  They also gave several
conjectures related to the Paving Conjecture.
Theorem \ref{MSS10} answers the {\bf Casazza-Tremain Conjecture}
\cite{CT}.

\begin{theorem}
Every unit norm 18-tight frame can be partitioned two subsets each of
which has frame bounds $2,16$.
\end{theorem}

\begin{proof}
Let $\{u_i\}_{i=1}^{18d}$ be a unit norm 18-tight frame in $\CC^d$.  By
Theorem \ref{MSS10}, we can find a partition $\{S_1,S_2\}$ of $[18d]$
so that for all $\|u\|=1$ we have
\[ \sum_{i\in S_1}|\langle u,u_i\rangle|^2 \le 16.\]
Thus,
\begin{align*}
 18 &= \sum_{i\in S_1}|\langle u,u_i\rangle|^2 + \sum_{i\in S_2}|\langle u,u_i\rangle|^2\\
&\le 16 + \sum_{i\in S_2}|\langle u,u_i\rangle|^2.
\end{align*}
It follows that
\[ \sum_{i\in S_2}|\langle u,u_i\rangle|^2 \ge 2.\]
By symmetry, 
\[ \sum_{i\in S_1}|\langle u,u_i\rangle|^2 \ge 2.\]
\end{proof}

In his work on time-frequency analysis, Feichtinger \cite{G,CT}
noted that all of the Gabor frames he was using 
had the property that they could be divided into a finite number
of subsets which were Riesz basic sequences.  This led to a 
conjecture known as the {\bf Feichtinger Conjecture} \cite{CCLV}.   
There is a significant body of work on this conjecture
\cite{BCHL,BCHL2,G} (See also \cite{L} for a large listing of
papers on the Feichtinger Conjecture in reproducing kernel Hilbert
spaces and many classical spaces such as Hardy space on
the unit disk, weighted Bergman spaces, and Bargmann-Fock spaces).
  The following theorem gives the best
quantative solution to
the {\bf Feichtinger Conjecture} from the results of \cite{MSS}.

\begin{theorem}\label{Fei2}
Every unit norm $B$-Bessel sequence can be partitioned into
$r$-subsets each of which is a $\epsilon$-Riesz basic sequence,
where
\[ r= \left ( \frac{6(B+1)}{\epsilon}\right )^4 \mbox{ in the real case } ,\]
and
\[ r= \left ( \frac{6(B+1)}{\epsilon}\right )^8 \mbox{ in the complex case } .\]
\end{theorem}

\begin{proof}
Fix $0< \epsilon <1$ and let $\{e_i\}_{i=1}^{\infty}$ be
an orthonormal basis for $\ell_2$.  Let
$T:\ell_2 \rightarrow \ell_2$ satisfy $\|Te_i\|=1$ for all
$i=1,2,\ldots$.  Let $S=T^*T$.  Since $S$ has ones on the
diagonal, $I-S$ has zero diagonal and so by Theorem \ref{thm6.1}
there is an
\[ r= \left ( \frac{6(\|s\|+1)}{\epsilon}\right )^4 \mbox{ in the real case } ,\]
and
\[ r= \left ( \frac{6(\|s\|+1)}{\epsilon}\right )^8 \mbox{ in the complex case } ,\]
and a partition $\{S_j\}_{j=1}^r$ of $\N$ so that if $Q_{S_j}$ is the diagonal
projection onto $\{e_i\}_{i\in S_j}$, we have
\[\|Q_{S_j}(I-S)Q_{S_j}\| \le \frac{\epsilon}{\|S\|+1}\|(I-S\|.\]
Now, for all $j=1,2,\ldots,r$ and all $u = \sum_{i\in S_j}a_ie_i$ we have:
\begin{align*}
\|\sum_{i\in S_j}a_iTe_i\|^2 &= \|TQ_{S_j}u\|^2\\
&= \langle TQ_{S_j}u,TQ_{S_j}u\rangle\\
&= \langle T^*TQ_{S_j}u,Q_{S_j}u\rangle\\
&= \langle Q_{S_j}u,Q_{S_j}u\rangle - \langle Q_{S_j}(I-S)Q_{S_j}u,Q_{S_j}u\rangle\\
&\ge \|Q_{S_j}u\|^2 - \frac{\epsilon}{\|S\|+1}\|I-S\|\|Q_{S_j}u\|^2\\
&\ge (1-\epsilon)\|Q_{S_j}u\|^2 \\
&= (1-\epsilon)\sum_{i\in S_j}|a_i|^2.
\end{align*}
Similarly,
\[ \|\sum_{i\in S_j}a_iTe_i\|^2 \le (1+\epsilon)\sum_{i\in S_j}|a_i|^2, \mbox{  for all }j=1,2,\ldots,r.\]
\end{proof}

This result also answers the {\bf Sundberg Problem} \cite{CKBook}.  The question was:  Can every bounded Bessel
sequence be written as the finite union of non-spanning sets?  The answer is now yes.  We just
partition our Bessel sequence into Riesz basic sequences.  It is clear that Riesz basic sequences
can be partitioned into non-spanning sets.  I.e.  Take one vector as one set and the rest of the
vectors as the other set.  Neither of these can span the Hilbert space.

This result answers another famous conjecture known as the {\bf $R_{\epsilon}$-Conjecture}.
This was introduced by Casazza/Vershynin (unpublished) and was shown to be equivalent to
the Paving Conjecture.
 Recall, if $\epsilon >0$ and $\{u_i\}_{i=1}^{\infty}$ is a unit norm Riesz bsic sequence with Riesz bounds
$A = 1-\epsilon, B=1+\epsilon$ we call $\{u_i\}_{i\in I}$ an
$\epsilon$-{\bf Riesz basic sequence}.  This is now a special case of the Feichtinger Conjecture,
Theorem \ref{Fei2}

\begin{theorem}\label{rep}
For every $0<\epsilon <1$ there is an $r\in \N$ so that every unit norm
Riesz basic sequence with upper Riesz bound $B$ is a finite union of
$\epsilon$-Riesz basic sequences, where $r$ is as in Theorem \ref{Fei2}.
\end{theorem}

We note that Theorem \ref{rep} fails for equivalent norms
on a Hilbert space.  For example, if we renorm ${\ell}_{2}$ by letting
$|\{a_i\}| = \|{a_i}\|_{{\ell}_2} +\sup_i |a_i|$, then the
$R_{\epsilon}$-Conjecture fails for this equivalent norm. To see this,
let $f_i = (e_{2i}+e_{2i+1})/(\sqrt{2}+1)$ where
$\{e_i\}_{i\in \N}$ is the unit vector basis of ${\ell}_2$.
This is now a
unit norm Riesz basic sequence, but no infinite subset satisfies theorem
\ref{rep}.  To check this, let
$J\subset \N$ with $|J|=n$ and $a_i = 1/\sqrt{n}$ for
$i\in J $.  Then,
$$
|\sum_{i\in J}a_i f_i| =  \frac{1}{\sqrt{2}+1}\left ( \sqrt{2} +
\frac{1}{\sqrt{n}}\right ).
$$
Since the norm above is bounded away from one for $n\ge 2$,
we cannot satisfy the requirements of theorem \ref{rep}.

In 1987, Bourgain and Tzafriri \cite{BT} proved a fundamental
result in Operator Theory known as the
{\bf Restricted Invertibility Principle}.

\begin{theorem}[Bourgain-Tzafriri]\label{TBST1}
There are universal constants $A,c>0$ so that whenever
$T:{\ell}_{2}^{n}\rightarrow {\ell}_{2}^{n}$ is a linear
operator for which $\|Te_{i}\|=1$, for $1\le i\le n$,
then there exists
a subset ${\sigma}\subset \{1,2,\ldots , n\}$ of cardinality
$|{\sigma}|\ge{cn}/{\|T\|^{2}}$ so that for all
$j=1,2,\ldots ,n$ and for all
choices of scalars $\{a_{j}\}_{j\in {\sigma}}$,
$$
\|\sum_{j\in {\sigma}}a_{j}Te_{j}\|^{2} \ge
A\sum_{j\in {\sigma}}|a_{j}|^{2}.
$$
\end{theorem}

In a significant advance, Spielman and Srivastave \cite{SS} gave
an algorithm for proving the restricted invertibility theorem.
Theorem \ref{TBST1} gave rise
to a problem in the area which has received a great deal of
attention \cite{BT1,CT} known as the {\bf Bourgain-Tzafriri
Conjecture}.  No one really noticed that this result is the finite
version of the {\bf Feichtinger Conjecture}.  This conjecture
is now a theorem.  The proof is identical to the proof of
Theorem \ref{Fei2}

\begin{theorem}\label{CBST1}
For every $0<\epsilon <1$ and 
for every $B>1$ there is a natural number $r$
satisfying:
For any natural number $n$,
if $T:{\ell}_2^n \rightarrow {\ell}_2^n$ is a linear operator
with $\|T\|\le B$ and $\|Te_{i}\|=1$ for all $i=1,2,\ldots , n$,
 then there is a partition
$\{S_{j}\}_{j=1}^{r}$ of $\{1,2,\ldots , n\}$ so that
for all $j=1,2,\ldots ,r$ and
 all choices of scalars $\{a_{i}\}_{i\in S_{j}}$ we have:
$$
(1-\epsilon) \sum_{i\in S_{j}}|a_{i}|^{2}\le
\|\sum_{i\in S_{j}}a_{i}Te_{i}\|^{2}\le (1+\epsilon) \sum_{i\in S_{j}}|a_{i}|^{2},
$$
where
\[ r= \left ( \frac{6(B+1)}{\epsilon)}\right )^4 \mbox{ in the real case } ,\]
and
\[ r= \left ( \frac{6(B+1)}{\epsilon}\right )^8 \mbox{ in the complex case } .\]
\end{theorem}

\section{Paving in Harmonic Anslysis}\label{HA}\label{HA1}

Recall the definition of Laurent Operator:

\begin{definition}
If $f \in L^{\infty}[0,1]$, the {\bf Laurent operator with symbol f}, denoted $M_f$,
is the operator of multiplication by $f$.
\end{definition}

In the 1980's, a very deep study of the Paving Conjecture for Laurant Operators
was carried out by Berman/Halpern/Kaftal/Weiss \cite{BHKW,BHKW2,HKW}.
They produced a parade of new techniques and interesting results in this direction
including the introduction of the notion of {\bf uniform paving}.  They also showed
that matrices with positive coefficients are pavable.

We need the following notation.
\vskip10pt
\noindent {\bf Notation}:  If $I\subset \Z$, we let $S(I)$ denote the $L^2([0,1])$-closure of the span of the exponential functions with frequencies taken from $I$:
\begin{displaymath}
S(I)=\mathrm{cl}(\mathrm{span}\{\mathrm{e}^{2\pi\mathrm{i}nt}\}_{n\in I}).
\end{displaymath}

A deep and fundamental question in Harmonic Analysis is to
understand the distribution of the norm of a function
$f\in S(I)$.  It is
known \cite{BHKW,BHKW2,HKW} if that if $[a,b] \subset [0,1]$ and $\epsilon >0$,
then there
is a partition of $\Z$ into arithmetic progressions
$A_j = \{nr+j\}_{n\in \Z}$, $0\le j\le r-1$
so that for all $f\in S(A_j)$ we have
$$
(1-\epsilon )(b-a)\|f\|^2
\le \|f\cdot {\chi}_{[a,b]}\|^2 \le (1+\epsilon)(b-a)\|f\|^2.
$$
What this says is that the functions in $S(A_j)$ have their norms nearly
uniformly distributed across $[a,b]$ and $[0,1]\setminus [a,b]$.  The
central question is whether such a result is true for arbitrary
measurable subsets of $[0,1]$ (but it is known that the partitions
can no longer be arithmetic progressions \cite{BS,HKW,HKW2}).
If $E$ is a measurable
subset of $[0,1]$,
 let $P_E$ denote the orthogonal projection
of $L^2[0,1]$ onto $L^2(E)$, that is, $P_E(f) = f\cdot {\chi}_E$.
The fundamental question here for many years, is now answered
by the following result which is an immediate consequence of
Theorem \ref{Fei2}.

\begin{theorem}\label{C100}
If $E\subset [0,1]$ is measurable and $\epsilon >0$ is given,
there is a partition $\{S_j\}_{j=1}^{r}$ of $\Z$
so that for all $j=1,2,\ldots ,r$ and all $f\in S(A_j)$
\[
(1-\epsilon )|E|\|f\|^2
\le \|P_{E}(f)\|^2 \le (1+\epsilon)|E|\|f\|^2,
\]
where
\[ r = \left ( \frac{6(|E|+1)}{\epsilon |E|} \right )^{8}.\]
\end{theorem}

Recall that $\{e^{2\pi int}\}_{n\in \Z}$ is an orthonormal basis for $L^2[0,1]$.
If $E \subset [0,1]$ of positive Lebesgue measure and $L^2(E)$ denotes the
corresponding Hilbert space of square-integrable functions on $E$, then
$f_n(t) = e^{2\pi int}\chi_E$ for ${n\in \Z}$ is a Parseval frame for $L^2(E)$ called
the {\bf Fourier Frame} for $L^2(E)$.  Much work has been expended on trying to
prove the Feichtinger Conjecture for Fourier Frames.

If $f\in L^2[0,1]$ and $0 \not= a$ we define
\[ (T_af)(t)= f(t-a).\]
Casazza/Christensen/Kalton \cite{CCK} showed that if $f\in L^2[0,1]$ then
$\{T_n(f)\}_{n\in \N}$ is a frame if and only if it is a Riesz basic sequence.

Halpern/Kaftal/Weiss \cite{HKW} studied {\bf uniform pavings} for Laurent operators.
In particular, they asked when we can pave Laurent operators with arithmetic
progressions from $\Z$?  They showed that this occurs if and only if the symbol
$f$ is Riemann Integrable.  As a consequence of \cite{MSS}, a result of Paulsen
implies that we can at least pave all Laurent operators with subsets of $\Z$ which
have {\bf bounded gaps}.

\begin{definition}
A set $S \subset \N$ is called {\bf syndetic} if for some finite subset $F$ of $\N$
we have
\[ \cup_{n\in F}(S-n)=\N ,\]
where 
\[ S-n = \{m\in \N:m+n \in S\}.\]
\end{definition}

Thus syndetic sets have {\bf bounded gaps}.   I.e.  There is an integer $p$ so that
$[a,a+1,\ldots,a+p] \cap S \not= \phi$ for every $a \in \N$.  We will call $p$ the
{\bf gap length}.

Paving by syndetic sets arose from the fact that the Grammian of a Fourier Frame
is a Laurent matrix.  Moreover, dividing frames into Riesz sequences is equivalent
to paving their Grammian \cite{P,P1,CFMT1}.   

At GPOTS (2008) Paulsen presented a quite general paving result which included
paving by syndetic sets.  The idea was to work in $\ell^2(G)$ where $G$ is a countable
discrete group and $G$-invariant frames - I.e. frames which are invariant
under the action of $G$.  Fourier frames are thus $\Z$-invariant.  Paulsen next shows
that a frame is G-invariant if and only if its Grammian belongs to the group von Neumann
algebra $VN(G)$.  Paulsen then shows that an element of $VN(G)$ is pavable if and
only if it is pavable by syndetic sets.  Unraveling the notation
, it follows that $G$-invariant frames which can be partitioned
into Riesz sequences can also be partitioned with syndetic partitions.  
These results then appeared in his paper \cite{P,P1}.   Lawton, \cite{La} gave a direct
proof of  syndetic pavings for Fourier Frames.  We now give this result with the
constants available from \cite{MSS}.    

\begin{theorem}
The Fourier frame $\{e^{2 \pi int}\chi_E\}_{n\in \Z}$ for $L^2(E)$ can be partitioned
into $r$ syndetic sets $\{S_j\}_{j=1}^r$ with gap length $p\le r$ so that
\[ \left \{e^{2\pi int}\chi_E\right \}_{n \in S_j}\mbox{ is a $\epsilon$-Riesz sequence for all }
j=1,2,\ldots,r,\] where
\[ r= \left (  \frac{6(|E|+1)}{\epsilon |E|} \right )^8.\]  
\end{theorem}

\section{"Large" and "Decomposable" subspaces of $\H$}\label{LDS}
\setcounter{equation}{0}

In this section we give some new theorems arising from \cite{MSS}
relating to {\bf large} and {\bf decomposable} subspaces
of a Hilbert space.  These ideas were introduced in \cite{CFTW}.
Throughout this section we will use the
notation:
\vskip10pt
\noindent {\bf Notation}:  If $E\subset I$ we let $P_E$ denote the
orthogonal projection of ${\ell}_2(I)$ onto ${\ell}_2(E)$.
Also, recall that we write $\{e_i\}_{i\in I}$ for the standard
orthonormal basis for ${\ell}_2(I)$.
\vskip10pt
For results on frames, see Section \ref{FT}.

\begin{definition}
A subspace $\H$ of ${\ell}_2(I)$ is {\bf A-large} for $A>0$
if it is closed and for each $i\in I$, there is a vector
$f_i \in \H$ so that $\|f_i\|=1$ and $|f_i(i)|\ge A$.  The
space $\H$ is {\bf large} if it is A-large for some $A>0$.
\end{definition}

It is known that every frame is isomorphic to a Parseval frame.  The next
lemma gives an alternative identification of these Parseval frames and relies 
on Proposition \ref{FTP10}.  

\begin{lemma}
 Let $T^{*}:\H \rightarrow
{\ell}_2(I)$ be the analysis operator for a frame
$\{f_i\}_{i\in I}$ for $\H$ and let $P$ be the orthogonal
projection of ${\ell}_2(I)$ onto $\H$.  Then $\{Pe_i\}_{i\in I}$
is a Parseval frame for $T^{*}(\H)$ which is isomorphic to $\{f_i\}_{i\in I}$.
\end{lemma}

{\it Proof}:
Note that $\{Pe_i\}_{i\in I}$ is a Parseval frame (Theorem \ref{T3})
with synthesis operator $P$ and analysis operator $T_{1}^{*}$ satisfying
$T_{1}^{*}(\H) = P({\ell}_2(I)) = T^{*}(\H)$.  By Proposition \ref{FTP10},
$\{Pe_i\}_{i\in I}$
is equivalent to $\{f_i\}_{i\in I}$.
\qed

Now we will relate large subspaces of a Hilbert space with the range of the
analysis operator of some bounded frame.  We also give a quantative version
of the result for later use.

\begin{proposition}\label{CEP2}
Let $\H$ be a subspace of ${\ell}_2(I)$.  

(I)  The following are equivalent:

(1)  The subspace $H$ is large.

(2)  The subspace $H$ is the range of the analysis operator of
some bounded frame.

(II)  The following are equivalent:

(3)  The subspace $\H$ is A-large.

(4)  If $P$ is the orthogonal projection of ${\ell}_2(I)$ onto
$\H$ then $\|Pe_i\|\ge A$, for all $i\in I$.

\end{proposition}

{\it Proof}:  (I)  
$(1)\Rightarrow (2)$:  Suppose $\H$ is large.  So, there exists an
$A>0$ such that for each $i\in I$, there exists a vector $f_i \in \H$
with $\|f_i\|=1$ and $|f_i(i)|\ge A$.  Given the projection $P$ of
(2) we have
$$
A \le |f_i(i)| = |\langle e_i,f_i\rangle | = |\langle Pe_i,f_i\rangle |
\le \|Pe_i\|\|f_i\| = \|Pe_i\|.
$$

Note that this also proves (II) (3) $\Rightarrow $ (4).

$(2) \Rightarrow (1)$:  Assume $\{f_i\}_{i\in I}$ is a bounded
frame for a Hilbert space $\K$ with analysis operator $T^{*}$ and
$T^{*}(\K) = \H$.  Now, $\{Pe_i\}_{i\in I}$ is a Parseval frame
for $\H$ which is the range of its own analysis operator.  Hence,
$\{f_i\}_{i\in I}$ is equivalent to $\{Pe_i\}_{i\in I}$ by
Proposition \ref{FTP10}.  Since $\{f_i\}_{i\in I}$ is bounded,
so is $\{Pe_i\}_{i\in I}$.  Choose $A>0$ so that $A\le \|Pe_i\| \le 1$,
for all $i\in I$.  Then
$$
A\le |\langle Pe_i,Pe_i \rangle | = |\langle Pe_i,e_i\rangle | =
|Pe_i (i)|.
$$
So $\H$ is a large subspace.

Note that this also proves (II) (4) $ \Rightarrow $ (3).
\qed

Now we need to learn how to decompose the range of the analysis operator
of our frames.

\begin{definition} \label{D:rD}
A closed subspace $\H$ of ${\ell}_2(I)$ is {\bf r-decomposable}
if for some natural number $r$ there exists a partition
$\{S_j\}_{j=1}^{r}$ of $I$
 so that $P_{S_j}(\H) = {\ell}_2(S_j)$, for all
$j=1,2,\ldots ,r$.  The subspace $\H$ is {\it finitely decomposable}
if it is r-decomposable for some r.
\end{definition}

For the next proposition we need a small observation.

\begin{lemma}\label{TFAL1}
Let $\{f_i\}_{i\in I}$ be a Bessel sequence in $\H$ having
synthesis operator $T$ and
analysis operator $T^{*}$.  Let
$E\subset I$, and let $\{f_i\}_{i \in E}$ have analysis operator $(T|_E)^{*}$.  Then
$$
P_ET^{*} = (T|_E)^{*}.
$$
\end{lemma}

{\it Proof}:
For all $f\in \H$,
$$
P_E T^{*}(f) = P_E\left ( \sum_{i\in I}\langle f,f_i\rangle e_i \right )
= \sum_{i\in E}\langle f,f_i \rangle e_i = (T|_E)^{*}(f).
$$
\qed

We now have

\begin{proposition}\label{CEP1}
If $\{f_i\}_{i\in I}$ is a unit norm frame for $\K$ with analysis operator $T^*$, then
for any $0<\epsilon <1$, $ T^{*}(\K)$ is r-decomposable  for
\[ r = \left ( \frac{6(\|T\|^2+1)}{\epsilon}\right )^4,\]
(with power 8 in the complex case).
\end{proposition}

{\it Proof}:
We can partition $I$ into $\{S_j\}_{j=1}^{r}$ so that each
$\{f_i\}_{i\in S_j}$ is a Riesz basic sequence where
\[ r = \left ( \frac{6(\|T\|^2+1)}{\epsilon}\right )^2,\]
for any $0<\epsilon <1$.
Thus,
(see the discussion after Theorem \ref{FTTT})
$(T|_{S_j})^{*}$ is onto for every
$j=1,2,\ldots ,r$ and hence (by Lemma \ref{TFAL1})
$P_{S_j}T^{*}$ is onto for all $j=1,2,\ldots ,r$.
\qed

Now we can put this altogether.

\begin{theorem}
For every $0<A<1$ and $0<\epsilon <1$, 
there is a natural number 
\[  r = \left ( \frac{6(A^2+1)}{\epsilon A^2}\right )^4,\]
(power 8 in the complex case)  
so that every A-large subspace of ${\ell}_2(I)$ is r-decomposable.
\end{theorem}

\begin{proof}

By Proposition \ref{CEP2}, if $P$ is the orthogonal projection of
$\ell_2(I)$ onto $H$, then $\|Pe_i\|\ge A$, for all $i\in I$.  Then
\[ \{f_i\}_{i\in I} = \left \{\frac{Pe_i}{\|Pe_i\|}\right \}_{i\in I},\]
is a unit norm frame with Bessel bound $1/A^2$.  So by
Proposition \ref{CEP1}, our subspace is $r$-decomposable for
\[ r = \left (\frac {6(A^2+1)}{\epsilon A^2}\right )^4,\]
for any $0<\epsilon <1$.   
\end{proof}

\section{Open Problems}

There are a number of important open problems which remain even after
the work of \cite{MSS}.

\begin{problem}
Can the $\eta$ and $\theta$ in Theorem \ref{MSS10} be improved?
\end{problem}

\begin{problem}
Can the values of $r$ in the various results be improved?
\end{problem}

It has been shown \cite{CCLV,CT} 
that every unit norm two tight frame can be partitioned into two
linearly independent sets.  But, there do not exist universal constants
$A,B$ so that all such frames can be partitioned into two subsets
each with Riesz bounds $A,B$ \cite{CEKP,CFMT1}.
This result raises a fundamental problem.

\begin{problem}
Can every unit norm two tight frame be partitioned into three subsets each
of which are Riesz basic sequences with Riesz bounds independent
of the dimension of the space?  
\end{problem}

Perhaps the most important open problem:

\begin{problem}
Find an implementable algorithm for proving the Paving Conjecture.
\end{problem}

The most important case is really finding an algorithm for
proving the Feichtinger Conjecture.  The Feichtinger Conjecture potentially
could have serious applications if we could quickly compute the appropriate
subsets which are Riesz basic sequences.

\section{Acknowledgement}\label{A}

For many years the Kadison-Singer Problem was a major motivating force for
many of us.  It's challenges made every day an exciting event.  It also brought
together mathematicians from many diverse areas of research - especially
as the "polynomial people" came in to give the solution.  
As we discovered more elementary formulations of the problem,
it became clear that this problem represented a fundamental idea for finite
dimensional Hilbert spaces which was not understood at all.  This just made the
problem even more interesting.  Almost everyone believed that the problem had
a negative answer - which probably contributed to the problem remaining open
for so long since we were only looking for a counter-example.  The solution to this
problem by Marcus/Spielman/Srivastave was a major achievement of our time
and earned them the {\bf Polya Prize}, a trip to the {\bf International Congress
of Mathematicians} and recognition yet to be established.

The 54 year search for a solution to the Kadison-Singer Problem represented a large number of
papers by many brilliant mathematicians which culminated in the solution to the problem by
Marcus/Spielman/Srivastave.  We enclose here a brief summary of the historical development
of the Kadison-Singer Problem from the direction of the Paving Conjecture.  Since the authors
are not experts in $C^*$-Algebras, we have chosen not to trace the history of the problem from
that direction.  So we will start in 1979 with the introduction of the {\bf Anderson Paving Conjecture}.
Also, there are hundreds of papers here so we will just consider those papers which introduced
new directions (equivalences) of the Paving Conjecture.
\begin{itemize}
\item \cite{KS} (1959)  Kadison and Singer formulate the {\bf Kadison-Singer Problem}.
\item \cite{A} (1979)  Anderson reformulates the Kadison-Singer Problem into the
{\bf Anderson Paving Conjecture}. This was significant because it changed the Kadison-Singer
Problem from being a specialized problem hidden in $C^*$-Algebras and opened it up to 
everyone in Analysis.

\item \cite{BHKW,BHKW2,HKW,HKW2} (1986) Berman/Halpern/Kaftal/Weiss make a deep study of the Paving Conjecture
for Laurant Operators.  They introduced the notion of {\bf uniform pavability} and they show that matrices with positive coefficients are {\bf pavable}.  They also give a positive solution for paving for the Schatten $C_p$-norms for $p=4,6$.

\item \cite{BT} (1989) Bourgain-Tzafriri prove the famous {\bf Restricted Invertibility Theorem} which 
naturally leads to the {\bf Bourgain-Tzafriri Conjecture}.

\item \cite{BT1} (1991)  Bourgain-Tzafriri show that matrices with small entries are {\bf pavable}.

\item \cite{AA} (1991)  Akemann and Anderson formulate the {\bf Akemann-Anderson Projection Paving Conjecture}
and show it implies a positive solution to the Kadison-Singer Problem.  This was important since it reduced
paving to paving for a much smaller class of operators - projections.

\item \cite{G}(2003) Grochenig shows that localized frames satisfy the {\bf Feichtinger Conjecture}.
This conjecture was formulated in \cite{CCLV} but appeared much later.

\item \cite{W,W2} (2003-2004)  Weaver formulates the {\bf Weaver Conjectures} and gives a counter-example
to a conjecture of Akemann and Anderson which would have implied a positive solution to the
Kadison-Singer Problem.

\item \cite{CCLV} (2005) Casazza/Christensen/Lindner/Vershynan introduce the \\{\bf Feichtinger Conjecture}
and show it is equivalent to the {\bf Bourgain-Tzafriri Conjecture}.

\item \cite{BS} (2006)  Bownik/Speegle make a detailed study of the Feichtinger Conjecture for Wavlet Frames,
Gabor Frames and Frames of Translates and relate the Feichtinger Conjecture to Gowers' work on a 
generalization of Van der Waerdan's Theorem.

\item \cite{CT,CFTW} (2006)  Casazza/Tremain and Casazza/Fickus/Tremain/Weber show that the Kadison-Singer
Problem is equivalent to the {\bf Feichtinger Conjecture}, the {\bf Bourgain-Tzafriri
Conjecture}, the
{\bf $R_{\epsilon}$-Conjecture}, the {\bf Casazza/Tremain Conjecture},
 and a number of conjectures in Time-Frequency Analysis, Frames
of Translates and Hilbert Space Frame Theory.

\item \cite{PR} (2008) Paulsen and Raghupathi show that {\bf paving} (respectively, paving Toeplitz operators)
is equivalent to paving upper trianguar matrices (respectively,
paving upper triangular Toeplitz operators).

\item \cite{CEKP} (2009)  Casazza/Edidin/Kalra/Paulsen show that paving projections with constant diagonal 1/2
is equivalent to the Paving Conjecture.  This is a new direction for paving projections as all
previous work involved paving projections with very small diagonals.
They also show that the Paving Conjecture fails for
2-paving.

\item \cite{La} (see also \cite{P}) (2010)  Lawton and Paulsen independently showed
 that if the Feichtinger Conjecture holds
for Fourier frames, then each set in the partition into Riesz basic sequences can be chosen
to be a {\bf syndetic set}.  Paulsen first presented this at GPOTS (2008).

\item (2007 - 2010) A large number of papers on the Feichtinger Conjecture appeared.  Too numerous
to list here.  See \cite{L} for a somewhat complete list - especially for reproducing kernel Hilbert spaces
and for classical spaces.

\item \cite{CFMT1} (2011) Casazza/Fickus/Mixon/Tremain give concrete constructions of non-2-pavable
projections.

\item \cite{Ca}  (2012) Casazza introduces the {\bf Sundberg Problem} which is implied by the Paving Conjecture.

\item \cite{MSS} (2013) Marcus/Spielman/Srivastava surprise the mathematical community by giving a positive
solution to the Kadison-Singer Problem.

\end{itemize}

\begin{remark}
We were recently made aware of the thesis \cite{l} of
Y. Lonke from 1993 which has a proof that BT is equivalent
to KS.  Since it was written Hebrew, it seems to have been
overlooked.  We now have English translations \cite{l}. 
\end{remark}

\end{document}